\documentclass[a4paper]{amsart}
\usepackage{amssymb,amscd}
\usepackage[cp1251]{inputenc}
\usepackage[T2A]{fontenc}
\usepackage[english]{babel}
\usepackage{graphicx}
\usepackage[all]{xy}
\usepackage{amsmath}
\usepackage{cancel}
\usepackage{enumitem}

\input xypic

\theoremstyle{problems}

\newtheorem{theorem}{Theorem}[section]
\newtheorem{proposition}[theorem]{Proposition}
\newtheorem{lemma}[theorem]{Lemma}
\newtheorem{corollary}[theorem]{Corollary}

\theoremstyle{definition}

\theoremstyle{remark}
\newtheorem*{remark}{remark}

\numberwithin{equation}{section}

\renewcommand{\phi}{\varphi}
\newcommand{\bl}{\mbox{$\lambda\kern-0.53em\lambda$}}
\newcommand{\bmu}{\mbox{$\mu\kern-0.55em\mu$}}
\newcommand{\bnu}{\mbox{$\nu\kern-0.51em\nu$}}
\def\bphi{\mbox{$\varphi\kern-0.59em\varphi$}}

\def\R{\mathbb R}

\def\Z{\mathbb Z}

\def\sK{\mathcal K}

\newcommand{\mb}[1]{{\textbf {\textit#1}}}

\renewcommand{\ge}{\geqslant}

\renewcommand{\leq}{\leqslant}

\def\лк{\symbol{"BE}}
\def\пк{\symbol{"BF}}

\newcommand{\rank}{\mathop{\mathrm{rank}}}

\def\raag{\mbox{\it RA\/}}
\def\racg{\mbox{\it RC\/}}

\newcommand{\rk}{\mathcal R_{\mathcal K}}

\def\pt{\mathit{pt}}

\begin{document}


\title[Graded components of $L(\racg_\sK)$]{Graded components of the Lie algebra associated with the lower central series of a right-angled Coxeter group}
\author{Y.~Veryovkin}

\address{HSE, Faculty of Computer Science, International Laboratory of Algebraic Topology and its Applications, Lomonosov MSU,
Faculty of Mechanics and Mathematics}
\email{verevkin\_j.a@mail.ru}

\thanks{The study was supported by the Russian Science Foundation grant No. 21-71-00049, https://rscf.ru/project/21-71-00049/}

\begin{abstract}
The lower central series of the rgiht-angled Coxeter group $\racg_\sK$ and the corresponding graded Lie algebra $L(\racg_\sK)$ associated with the lower central series of a right-angled Coxeter group are studied. Relations are obtained in the graded components of the Lie algebra $L(\racg_\sK)$. A basis of the fourth graded component $L(\racg_\sK)$ for groups with at most $4$ generators was described.
\end{abstract}

\maketitle

\section{Introduction}
The right-angled Coxeter group $\racg_\sK$ is a group with $m$ generating $g_1,\ldots,g_m$ relations $g_i^2=1$ for all $i \in \{1,\ldots,m\}$, as well as the relations $g_ig_j=g_jg_i$ for some pairs $\{i,j\}$. Such a group can be defined by a graph $\sK^1$ with $m$ vertices, where if the corresponding generators commute, then a pair of vertices is connected by an edge. Also of interest is the lower central series of the group, with the help of which the Lie algebra associated with group is constructed. This paper is devoted to the description of the Lie algebra associated with the right-angled Coxeter group (associated Lie algebra).

For the right-angled Artin groups $\raag_\sK$ (which differ from the right-angled Coxeter groups $\racg_\sK$ by the absence of the relations $g_i^2 = 1$), the associated Lie algebra $L(\raag_\sK)$ was completely calculated in~ \cite{Duch-Krob}, see also~\cite{WaDe},~\cite{Papa-Suci}. Namely, an isomorphism between the Lie algebra $L(\raag_\sK)$ and the Lie graph algebra (over $\mathbb Z$) corresponding to the graph $\sK^1$ was proved.

In the case of right-angled Coxeter groups, in some special cases, factor groups $\gamma_1(\racg_\sK) / \gamma_n(\racg_\sK)$ were studied for some $n$ (see~works~\cite{Struik1}, ~\cite{Struik2}). For right-angled Coxeter groups, in contrast to right-angled Artin groups, the problem of describing the associated Lie algebra $L(\racg_\sK)$ is much more difficult, since there is no isomorphism between the algebra $L(\racg_\sK)$ and the Lie graph algebra $L_ \sK$ over $\mathbb Z_2$ (see~\cite[Example 4.3]{ve3}). In this paper~\cite{ve3} an epimorphism of Lie algebras $L_\sK \rightarrow L(\racg_\sK)$ is constructed and in some cases its kernel is described, and for an arbitrary group $\racg_\sK$ a combinatorial description of the bases of the first three graded components of the Lie algebra.

In this paper~\cite{WaldingerLCS} the dimensions of the successive factors of the members of the lower central series are calculated and their basis is constructed for the free product of direct sums of cyclic groups of order $2$, which is a subset of right-angled Coxeter groups. We construct a basis for the $4$-th graded component of the associated Lie algebra for Coxeter groups with $3$ and $4$ generators (the associated Lie algebra is completely described for $2$ generators, see~\cite[Proposition 4.4]{ve3}). Unlike the bases constructed in~\cite{WaldingerLCS}, the bases constructed in this paper consist entirely of simple nested commutators.

The author is the winner of the ``Young Mathematics of Russia'' competition in 2019. The author expresses his deep gratitude to his supervisor Taras Evgenievich Panov for posing the problem, constant attention and assistance in the work.


\section{Preliminaries}
Let $\sK$ be an (abstract) simplicial complex on the set $ [m] = \{1,2, \dots, m \}$. 
A subset $I=\{i_1,\ldots,i_k\}\in\mathcal K$ is called \emph{a simplex} (or \emph{face}) of~$\sK$. We always assume that $\sK$ contains $\varnothing$ and all singletons $\{i\}$, $i = 1, \ldots, m$.

We denote by $F_m$ or $F(g_1, \ldots, g_m)$ a free group of rank $m$ with generators $g_1, \ldots, g_m$.

The \emph{right-angled Coxeter (Artin) group} $\racg_\sK$ ($\raag_\sK$) corresponding to~$\sK$ is defined by generators and relations as follows:
\[
  \racg_\sK = F(g_1,\ldots,g_m)\big/ (g_i^2 = 1 \text{ for } i \in \{1, \ldots, m\}, \; \; g_ig_j=g_jg_i\text{ when
  }\{i,j\}\in\sK),
\]
\[
  \raag_\sK = F(g_1,\ldots,g_m)\big/ (g_ig_j=g_jg_i\text{ when
  }\{i,j\}\in\sK).
\]

Clearly, the group $\racg_\sK$ ($\raag_\sK$) depends only on the $1$-skeleton of~$\sK$, the graph~$\sK^1$.

We recall the construction of polyhedral products.

Let $\sK$ be a simplicial complex on~$[m]$, and let
\[
  (\mb X,\mb A)=\{(X_1,A_1),\ldots,(X_m,A_m)\}
\]
be a sequence of $m$ pairs of pointed topological spaces, $\pt\in A_i\subset X_i$, where $\pt$ denotes the basepoint. For each subset $I\subset[m]$ we set
\begin{equation}\label{XAI}
  (\mb X,\mb A)^I=\bigl\{(x_1,\ldots,x_m)\in
  \prod_{k=1}^m X_k\colon\; x_k\in A_k\quad\text{for }k\notin I\bigl\}
\end{equation}
and define the \emph{polyhedral product} $(\mb X,\mb
A)^\sK$ as
\[
  (\mb X,\mb A)^{\sK}=\bigcup_{I\in\mathcal K}(\mb X,\mb A)^I=
  \bigcup_{I\in\mathcal K}
  \Bigl(\prod_{i\in I}X_i\times\prod_{i\notin I}A_i\Bigl),
\]
where the union is taken inside the Cartesian product $\prod_{k=1}^m X_k$.

In the case when all pairs $(X_i,A_i)$ are the same, i.\,e.
$X_i=X$ and $A_i=A$ for $i=1,\ldots,m$, we use the notation
$(X,A)^\sK$ for $(\mb X,\mb A)^\sK$. Also, if each $A_i=\pt$, then
we use the abbreviated notation $\mb X^\sK$ for $(\mb X,\pt)^\sK$,
and $X^\sK$ for $(X,\pt)^\sK$.

For details on this construction and examples see~\cite[\S3.5]{bu-pa00},~\cite{b-b-c-g10},~\cite[\S4.3]{bu-pa15}.

Let $(X_i,A_i)=(D^1,S^0)$ for $i = 1, \ldots, m$, where $D^1$ is a segment and $S^0$ is its boundary consisting of two points. The corresponding polyhedral product is known as the \emph{real moment-angle complex}~\cite[\S3.5]{bu-pa00},~\cite{bu-pa15} and is
denoted by~$\rk$:
\begin{equation}\label{rk}
  \rk=(D^1,S^0)^\sK=\bigcup_{I\in\sK}(D^1,S^0)^I.
\end{equation}

We shall also need the polyhedral product $(\R P^\infty)^\sK$, where $\R P^\infty$ the infinite-dimensional real projective space.

A simplicial complex $\sK$ is called a \emph{flag complex} if any set of vertices of $\sK$ which are pairwise connected
by edges spans a simplex. Any flag complex $\sK$ is determined by its one-dimensional skeleton $\sK^1$.

The relationship between polyhedral products and right-angled Coxeter groups is described by the following result.

\begin{theorem}[{see \cite[Corollary~3.4]{pa-ve}}]\label{coxfund}
Let $\sK$ be a simplicial complex on $m$ vertices.
\begin{itemize}
\item[(a)] $\pi_1((\R P^\infty)^\sK)\cong\racg_\sK$.
\item[(b)] Each of the spaces $(\R P^\infty)^\sK$ и $\rk$ is aspherical if and only if $\sK$ is a flag complex.
\item[(c)] $\pi_i((\R P^\infty)^\sK)\cong\pi_i(\rk)$ for $i\ge2$.
\item[(d)] The group $\pi_1(\rk)$ isomorphic to the commutator subgroup~$\racg'_\sK$.
\end{itemize}
\end{theorem}

For each subset~$J\subset[m]$, consider the restriction of $\sK$ to~$J$:
\[
  \sK_J=\{I\in\sK\colon I\subset J\},
\]
which is also known as a \emph{full subcomplex} of~$\sK$.

The following theorem gives a combinatorial description of homology of the real moment-angle complex $\mathcal R_\sK$:

\begin{theorem}[{\cite{bu-pa00}, \cite[\S4.5]{bu-pa15}}]\label{homrk} There is an isomorphism
\[
  H_k(\rk;\Z)\cong\bigoplus_{J\subset[m]}\widetilde
  H_{k-1}(\sK_J)
\]
for any $k \ge 0$, where $\widetilde H_{k-1}(\sK_J)$~is the reduced simplicial homology group of ~$\sK_J$.
\end{theorem}

If $\sK$ is a flag complex, then Theorem~\ref{homrk} also gives a description of the integer homology groups of the commutator subgroup $\racg_\sK'$.



\medskip
Let $G$ be group. The \emph{commutator} of two elements $a, b \in G$ given by the formula $(a,b) = a^{-1}b^{-1}ab$.

We refer to the following nested commutator of length $k$
$$
(q_{i_1}, q_{i_2}, \ldots, q_{i_k}) := (\ldots((q_{i_1}, q_{i_2}), q_{i_3}), \ldots, q_{i_k}).
$$
as the \emph{simple nested commutator} of $q_{i_1}, q_{i_2}, \ldots, q_{i_k}$.

Similarly, we define \emph{simple nested Lie commutators}
$$
[\mu_{i_1}, \mu_{i_2}, \ldots, \mu_{i_k}] := [\ldots[[\mu_{i_1}, \mu_{i_2}], \mu_{i_3}], \ldots, \mu_{i_k}].
$$

For any group $G$ and any three elements $a, b, c \in G$, the following \emph{Hall--Witt identities} hold:
\begin{equation}\label{WH}
\begin{aligned}
&(a, bc) = (a, c) (a, b) (a, b, c),\\
&(ab, c) = (a, c) (a, c, b) (b, c),\\
&(a,b,c)(b,c,a)(c,a,b)=(b,a)(c,a)(c,b)^a(a,b)(a,c)^b(b,c)^a(a,c)(c,a)^b,
\end{aligned}
\end{equation}
where $a^b = b^{-1}ab$.

Let $H, W \subset G$ be subgroups. Then we define $(H, W) \subset G$ as the subgroup generated by all commutators $(h, w), h \in H, w \in W$. In particular, the \emph{commutator subgroup} $G'$ of the group $G$ is $(G, G)$.

For any group $G$, set $\gamma_1(G) = G$ and define inductively $\gamma_{k+1}(G) = (\gamma_{k}(G), G)$. The resulting sequence of groups $\gamma_1(G), \gamma_2(G), \ldots, \gamma_k(G), \ldots$ is called the \emph{lower central series} of $G$.

If $H \subset G$ is normal subgroup, i.\,e. $H = g^{-1}Hg$ for all $g \in G$, we will use the notation $H \lhd G$.

In particular, $\gamma_{k+1}(G) \lhd \gamma_k(G)$, and the quotient group $\gamma_{k}(G) / \gamma_{k+1}(G)$ is abelian. Denote $L^k (G) := \gamma_{k}(G) / \gamma_{k+1}(G)$ and consider the direct sum
$$
L(G) := \bigoplus_{k=1}^{+\infty} L^k (G).
$$
Given an element $a_k \in \gamma_k(G) \subset G$, we denote by $\overline{a}_k$ its conjugacy class in the quotient group~$L^k (G)$. If $a_k \in \gamma_k(G), \; a_l \in \gamma_l(G)$, then $(a_k, a_l) \in \gamma_{k+l}(G)$. Then the Hall--Witt identities imply that $L(G)$ is a graded Lie algebra over $\mathbb Z$ (a Lie ring) with Lie bracket $[\overline{a}_k, \overline{a}_l] := \overline{(a_k, a_l)}$. The Lie algebra $L(G)$ is called the \emph{Lie algebra associated with the lower central series} (or the \emph{associated Lie algebra}) of $G$.

\begin{theorem}[{\cite[Theorem 4.5]{pa-ve}}]\label{gscox}
Let $\racg_\sK$ be right-angled Coxeter group corresponding to a simplicial complex~$\sK$ with $m$ vertices.
Then the commutator subgroup $\racg'_\sK$ has a finite minimal set of generators consisting of $\sum_{J\subset[m]}\rank\widetilde H_0(\sK_J)$ nested commutators
\begin{equation}\label{commuset}
  (g_i,g_j),\quad (g_{i},g_j,g_{k_1}),\quad\ldots,\quad
  (g_{i},g_{j},g_{k_1},g_{k_2},\ldots,g_{k_{m-2}}),
\end{equation}
where $i < j > k_1 > k_2 > \ldots > k_{\ell-2}$, $k_s\ne i$ for all~$s$, and
$i$~is the smallest vertex in a connected component not containing~$j$ of the subcomplex $\sK_{\{k_1,\ldots,k_{\ell-2},j,i\}}$.
\end{theorem}

\begin{remark}
In~\cite{pa-ve} commutators were nested to the right. Now we nest them to the~left.
\end{remark}

From Theorems~\ref{homrk} and \ref{gscox} we get:
\begin{corollary}\label{h1rk}
The group $H_1(\rk) = \racg_\sK' / \racg_\sK''$ is a free abelian group of rank $\sum_{J\subset[m]}\rank\widetilde H_0(\sK_J)$ with basis consisting of the images of the iterated commutators described in Theorem~\ref{gscox}.
\end{corollary}

The following standard result holds (see \cite[\S5.3]{Ma-Car-Sol}):

\begin{proposition}\label{comb}
Let $G$ be a group with generators $g_i, i \in I$. The terms of the lower central series $\gamma_k(G)$ are generated by simple nested commutators of length greater than or equal to $k$ from the generators and their inverses.
\end{proposition}

\begin{proposition}[{\cite[Proposition 3.3]{ve3}}]\label{kv}
The square of any element of $\gamma_k(\racg_\sK)$ is contained in $\gamma_{k+1}(\racg_\sK)$.
\end{proposition}

There are also the following results:

\begin{proposition}[{\cite[(4.19)]{WaldingerLCS}}]\label{numbergens3}
Let $\sK$ be a discrete set of $3$ points, i.e. $\racg_\sK = \mathbb Z_2\langle g_1 \rangle \ast \mathbb Z_2\langle g_2 \rangle \ast \mathbb Z_2\langle g_3 \rangle$. Then $L^4(\racg_\sK)$ has a minimal set of $8$ generators.
\end{proposition}

\begin{proposition}[{\cite[(4.19)]{WaldingerLCS}}]\label{numbergens3oneedge}
Let $\sK$ be a simplicial complex on $3$ points with a single edge, i.e. $\racg_\sK = (\mathbb Z_2\langle g_1 \rangle \oplus \mathbb Z_2\langle g_2 \rangle) \ast \mathbb Z_2\langle g_3 \rangle$. Then $L^4(\racg_\sK)$ has a minimal set of $4$ generators.
\end{proposition}




\begin{proposition}[{\cite[(4.19)]{WaldingerLCS}}]\label{numbergens4}
Let $\sK$ be a discrete set of $4$ points, i.e. $\racg_\sK = \mathbb Z_2\langle g_1 \rangle \ast \mathbb Z_2\langle g_2 \rangle \ast \mathbb Z_2\langle g_3 \rangle \ast \mathbb Z_2\langle g_4 \rangle$. Then $L^4(\racg_\sK)$ has a minimal set of $32$ generators.
\end{proposition}

\begin{theorem}[{\cite[Theorem 4.5]{ve3}}]\label{LRCK}
Let $\sK$ be a simplicial complex on $[m]$, let $\racg_\sK$ be the right-angled Coxeter group corresponding to $\sK$, and $L(\racg_\sK)$ its associated Lie algebra. Then:
\begin{itemize}
\item[(a)] $L^1(\racg_\sK)$ has a basis $\overline{g}_1, \ldots, \overline{g}_m$;
\item[(b)] $L^2(\racg_\sK)$ has a basis consisting of the commutators $[\overline{g}_i, \overline{g}_j]$ with $i < j$ and ${\{i, j\} \notin \sK}$;
\item[(c)] $L^3(\racg_\sK)$ has a basis consisting of
\begin{itemize}
\item[--] the commutators $[\overline{g}_i, \overline{g}_j, \overline{g}_j]$ with $i < j$ and $\{i, j\} \notin \sK$;
\item[--] the commutators $[\overline{g}_i, \overline{g}_j, \overline{g}_k]$ where $i < j > k, i \neq k$ and $i$ is the smallest vertex in a connected component of $\sK_{\{i,j,k\}}$ not containing~$j$.
\end{itemize}
\end{itemize}
\end{theorem}

\section{Main results}

Let $G$ be a group. The commutator of elements $a$ and $b$ is given by the formula $(a, b) = a^{-1}b^{-1}ab$.

Consider a simplicial complex --- a discrete set of $m$ points. We have $m$ generators $\mu_i$, each corresponding to its own $\mathbb Z_2$ in the direct sum of $m$ instances of $\mathbb Z_2$ (this is the one-dimensional graded component of $L(RC_\mathcal K)$). Let $FL(\mu_1,\ldots,\mu_m)$ be a free graded Lie algebra over $\mathbb Z_2$ with generators $\mu_1,\ldots,\mu_m$ of dimension~$1$. We have $L^1(RC_\sK) \cong FL^1\langle \mu_1, \ldots, \mu_m \rangle$, $L^2(RC_\sK) \cong FL^2\langle \mu_1, \ldots, \mu_m \rangle$ (since there are only commutators of length $2$ that are described in theorem~\ref{LRCK}), $L^3(RC_\sK) \cong FL^3\langle \mu_1, \ldots, \mu_m \rangle / ( [\mu_i, \mu_j, \mu_j]=[\mu_i, \mu_j, \mu_i] \; \forall i < j )$. The latter follows from Theorem~\ref{LRCK} and the following relation in an arbitrary right-angled Coxeter group:
\begin{equation}\label{equalinRC}
(g_i, g_j, g_i) = (g_i, g_j, g_j).
\end{equation}
This relation is verified directly:
$$
(g_i, g_j, g_i)=(g_j, g_i)g_i(g_i, g_j)g_i=g_jg_ig_jg_ig_ig_ig_jg_ig_jg_i=g_jg_ig_jg_ig_jg_ig_jg_i=(g_j, g_i)^2,
$$
$$
(g_i, g_j, g_j)=(g_j, g_i)g_j(g_i, g_j)g_j=g_jg_ig_jg_ig_jg_ig_jg_ig_jg_j=g_jg_ig_jg_ig_jg_ig_jg_i=(g_j, g_i)^2.
$$

In what follows, we assume that $\mu_i = \overline{g_i} = [g_i]$ is the adjacency class of an element of the group~$g_i$.

\begin{proposition}
Let $\sK$ be a discrete set of $m$ points, i.e. $\racg_\sK = \mathbb Z_2\langle g_1 \rangle \ast \ldots \ast \mathbb Z_2\langle g_m \rangle$. Consider the Lie algebra
$$
F = FL\langle \mu_1, \ldots, \mu_m \rangle / ( [\mu_i, \mu_j, \mu_j]=[\mu_i, \mu_j, \mu_i] \; \text{ для любого } i < j ).
$$
Then $F^i \cong L^i(RC_\sK)$ for $i \leq 3$.
%
\end{proposition}
\begin{proof}
Follows from~\cite[Proposition 4.4]{ve3}.
\end{proof}

\begin{proposition}\label{relationinL4}
Let $\sK$ be a discrete set of $m$ points, i.e. $\racg_\sK = \mathbb Z_2\langle g_1 \rangle \ast \ldots \ast \mathbb Z_2\langle g_m \rangle$. Consider the Lie algebra
$$
F = FL\langle \mu_1, \ldots, \mu_m \rangle / ( [\mu_i, \mu_j, \mu_j]=[\mu_i, \mu_j, \mu_i] \; \text{ для любого } i < j ).
$$
Then $F^4$ and $L^4(\racg_\sK)$ have the relations
$$
[\mu_i, \mu_j, \mu_i, \mu_i] = [\mu_i, \mu_j, \mu_i, \mu_j].
$$
\end{proposition}
\begin{proof}
Let's expand the commutators in the group $\racg_\sK$:
$$(i,j,i,i)=((j,i)^2,i)=ijijijijijijijijii=ijijijijijijijij=(i,j)^4;$$
$$(i,j,i,j)=((j,i)^2,j)=ijijijijjjijijijij=ijijijijijijijij=(i,j)^4,$$
and therefore the corresponding elements in $L^4(\racg_\sK)$.

Now consider the algebra $F$, in which we have:
$$[\mu_i, \mu_j, \mu_i, \mu_i] = [\mu_i, \mu_j, \mu_j, \mu_i] = [[\mu_i,[\mu_i,\mu_j]],\mu_j] + [[\mu_i,\mu_j],[\mu_j,\mu_i]] = [\mu_i, \mu_j, \mu_i, \mu_j].$$
\end{proof}

\begin{lemma}\label{comtogPV}
For any group $G$ and for any $a, b, c \in G$ it is true that
\begin{equation}
(a,(b,c))\!=\!(a,c)(c,(b,a))(a,b)(c,b)(b,(a,c))(c,a)(b,a)(b,c).
\end{equation}
\begin{equation}
((a,b),c)\!=\!(b,a)(c,a)(c,b)((c,b),a)(a,b)(a,c)((a,c),b)(b,c).
\end{equation}
\end{lemma}
\begin{proof}
The first identity is proved by directly expanding commutators and checking, while the second is proved by taking the inverse of the first one and replacing $b$ with $a$, $c$ with $b$, and $a$ with $c$ (the second can also be verified by expanding commutators).
\end{proof}

\begin{corollary}\label{comtogPVbM}
For any group $G$ and for any $a \in \gamma_i(G), b \in \gamma_j(G)$, $c~\in~\gamma_k(G)$ it is true that
\begin{equation}
(a,(b,c))\!\equiv\!(c,(b,a))(b,(a,c))\mod\gamma_{\min(2i+j+k, i+2j+k, i+j+2k)}(G).
\end{equation}
\begin{equation}\label{comtogPVbM2t}
((a,b),c)\!\equiv\!((c,b),a)((a,c),b)\mod\gamma_{\min(2i+j+k, i+2j+k, i+j+2k)}(G).
\end{equation}
\end{corollary}
\begin{proof}
Since $qw=wq(q,w)$, then for $q \in \gamma_{i_1}(G), w \in \gamma_{i_2}(G)$ we have $qw\equiv wq\mod\gamma_ {i_1+i_2}(G)$ since $(q, w) \in \gamma_{i_1+i_2}(G)$.
We have $(a,b) \in \gamma_{i+j}(G), (a,c)\in\gamma_{i+k}(G), (b,c)\in\gamma_{j+ k}(G)$, and any nested triple commutator containing $a, b, c$ belongs to $\gamma_{i+j+k}(G)$. We get that $(a,b)(b,c)\equiv(b,c)(a,b)\mod\gamma_{i+2j+k}(G)$, $(a,b)(a ,c)\equiv(a,c)(a,b)\mod\gamma_{2i+j+k}(G)$, $(b,c)(a,c)\equiv(a,c)( b,c)\mod\gamma_{i+j+2k}(G)$. Also, if $A$ is any triple nested commutator of $a, b, c$, then we have $(a,b)A\equiv A(a,b)\mod\gamma_{2i+2j+k}( G)$, $(a,c)A\equiv A(a,b)\mod\gamma_{2i+j+2k}(G)$, $(b,c)A\equiv A(b,c) \mod\gamma_{i+2j+2k}(G)$. Using this, we rearrange the commutators in the identities from the lemma~\ref{comtogPV}, canceling all double commutators with their inverses. From this the equalities to be proved follow.
\end{proof}

\begin{proposition}\label{obrcom}
For any $a \in \gamma_k(G), b \in \gamma_m(G), c \in \gamma_n(G)$ it is true that $(a, b, c)^{-1} \equiv (b, a, c) \mod \gamma_{k+m+n+1}(G)$.
\end{proposition}

\begin{proof}
Consider the transformations:
$$(a, b, c)^{-1} = (c, (a, b)) \equiv^{(1)} ((b, c), a)((c, a), b) \mod \gamma_{k + m + n + 1}(G) \equiv^{(2)}$$
$$\equiv^{(2)} ((c, a), b)((b, c), a) \mod \gamma_{2(k + m + n)}(G) \equiv^{(3)}$$
$$\equiv^{(3)} ((b, a), c) \mod \gamma_{\min(2k+m+n, k+2m+n, k+m+2n)}(G),$$
here $(1)$ follows from the Witt-Hall identity taken modulo the term of the lower central series, $(2)$ follows from the interchange of commutators, $(3)$ follows from corollary~\ref{comtogPVbM}. Since $k+m+n+1$ is the smallest number among the members of the lower central row in the chain, modulo comparison is taken from it.
\end{proof}

\begin{corollary}
For any $a \in \gamma_k(G), b \in \gamma_m(G), c \in \gamma_n(G)$ it is true that $(a, b, c) \equiv (c, (b, a)) \mod \gamma_{k+m+n+1}(G)$.
\end{corollary}

\begin{proof}
From the Proposition~\ref{obrcom} we get
$$(c, (b, a)) = (b, a, c)^{-1} \equiv ((b, a)^{-1}, c) \mod \gamma_{k+m+n+1}(G) = (a, b, c).$$
\end{proof}

\begin{proposition}\label{freeproduct3}
Let $\sK$ be a discrete set of $3$-th points, i.e. $\racg_\sK = \mathbb Z_2\langle g_1 \rangle \ast \mathbb Z_2\langle g_2 \rangle \ast \mathbb Z_2\langle g_3 \rangle$. Then $L^4(\racg_\sK)\cong\mathbb Z_2^8$ has the following minimum set of generators:
$$
[\mu_1, \mu_2, \mu_1, \mu_1], [\mu_1, \mu_3, \mu_1, \mu_1], [\mu_2, \mu_3, \mu_2, \mu_1], [\mu_1, \mu_3, \mu_2, \mu_1], [\mu_1, \mu_3, \mu_1, \mu_2],
$$
$$
[\mu_2, \mu_3, \mu_2, \mu_2], [\mu_2, \mu_3, \mu_1, \mu_2], [\mu_1, \mu_3, \mu_2, \mu_3].
$$
\end{proposition}

\begin{proof}

Consider the group $\gamma_1(\racg_\sK) /\gamma_2(\racg_\sK)$. By proposition~\ref{comb} its generators are $g_1, g_2, g_3$. Hence, from the proposition~\ref{comb} and the identities~\eqref{WH}, we obtain that the generators of the group $\gamma_2(\racg_\sK) / \gamma_3(\racg_\sK)$ are commutators of the form $(g_i, g_j) $, where $i, j \in \{1, 2, 3\}$. Select a set of commutators $A = \{(g_1, g_2), (g_1, g_3), (g_2, g_3)\}$, it generates all $\gamma_2(\racg_\sK) / \gamma_3(\racg_\sK) $, since all others are excluded as their inverses. Similarly, the group $\gamma_3(\racg_\sK) / \gamma_4(\racg_\sK)$ is generated by commutators of the form $(z, g_i)$, where $z \in A, i \in \{1, 2, 3 \}$. Since $(g_i, g_j, g_i) = (g_i, g_j, g_j)$ (see~identity~\eqref{equalinRC}), we can eliminate three commutators, plus one more, using the identities~\eqref{WH }. From the proposition~\ref{comb}, the identities~\eqref{WH} and the theorem~\ref{LRCK} we get that the group $\gamma_3(\racg_\sK) / \gamma_4(\racg_\sK)$ has a minimal set from $5$ generators $(g_1, g_2, g_1)$, $(g_1, g_3, g_1)$, $(g_1, g_3, g_2)$, $(g_2, g_3, g_1)$, $(g_2, g_3, g_2)$. We get that a set of $15$ commutators
$$
(g_1, g_2, g_1, g_1), (g_1, g_3, g_1, g_1), (g_2, g_3, g_2, g_1), (g_1, g_3, g_2, g_1), (g_2, g_3, g_1, g_1),
$$
$$
(g_1, g_2, g_1, g_2), (g_1, g_3, g_1, g_2), (g_2, g_3, g_2, g_2), (g_1, g_3, g_2, g_2), (g_2, g_3, g_1, g_2),
$$
$$
(g_1, g_2, g_1, g_3), (g_1, g_3, g_1, g_3), (g_2, g_3, g_2, g_3), (g_1, g_3, g_2, g_3), (g_2, g_3, g_1, g_3)
$$
generates $\gamma_4(\racg_\sK) / \gamma_5(\racg_\sK)$, but the constructed set of generators $L^4(\racg_\sK)$, corresponding to it is not the minimum on the proposal~\ref{numbergens3}. In what follows, for convenience, we will use the notation $(g_{i_1}, g_{i_2}, \ldots, g_{i_k}) = (i_1i_2i_3\ldots i_k)$, and also instead of $\gamma_k(\racg_\sK)$ we will write $\gamma_k$.

We have $(ijik)=((ji)(ji), k)=(jik)(((ji),k),(ji))(jik)\equiv(jik)^2 \mod \gamma_5$, отсюда $(1211)\equiv(211)^2\mod\gamma_5, (1212)\equiv(212)^2\mod\gamma_5$. We get that $(1211)\equiv(1212)\mod\gamma_5$ due to equality $(iji)=(ijj)$. Similarly, $(1311)\equiv(1313)\mod\gamma_5$ и $(2322)\equiv(2323)\mod\gamma_5$.

Из каждого из первых трёх столбцов мы убрали по одному коммутатору.

Consider three commutators $(1213), (2321)$ и $(1312)$. Notice, that $(2321) \equiv (321)^2 \mod \gamma_5$ и $(1312) \equiv (312)^2 \mod \gamma_5$.

Consider $(1213) \equiv (213)^2 \mod \gamma_5 \equiv^{(1)} (321)^{-2}(132)^{-2} \mod \gamma_5 =$

$= (321)^{-2}(2, (13))^2 \equiv^{(2)} (321)^{-2}(312)^{-2} \mod \gamma_5 \equiv (2321)(1312)^{-1} \mod \gamma_5$. Here $(1)$ follows from the third formula~\eqref{WH}, which requires additional explanations: since by permuting commutators of length $2$ we get commutators of length $4$, they are not equal to $1$ modulo $\gamma_5$, but since we have a commutator square, after all the necessary transformations modulo $\gamma_8 \subset \gamma_5$ we get squares of commutators of length $4$, which means, by Proposition~\ref{kv}, they lie in $\gamma_5$ and are equal to $1$ modulo $\gamma_5$. Comparison modulo $(2)$ follows from the proposition~\ref{obrcom} for similar reasons. From here we can also remove one of the commutators from the ``side'' diagonal, leaving any two of $(1213), (2321), (1312)$. Remove $(1213)$.

Note that $(2313) \equiv (2133)(1323) \mod \gamma_5$, which follows from applying the third identity~\eqref{WH} for (231) inside the commutator and then expanding it using the second identity~\eqref {WH}. Consider
$$(2133) \equiv (3123)(2313) \mod \gamma_5 \equiv^{(1)} (3123)((31),(23))(2331) \mod \gamma_5 \equiv^{(2)}$$
$$\equiv^{(2)} ((23),(13))^{-1}(1332)^{-1}((31),(23))(2331) \mod \gamma_5 \equiv$$
$$\equiv (1332)^{-1}(2331)((13),(23))((31),(23)) \mod \gamma_5 \equiv^{(3)}$$
$$\equiv^{(3)} (1332)^{-1}(2331)((31),(23))^{-1}((31),(23)) \mod \gamma_5 \equiv (1332)^{-1}(2331) \mod \gamma_5,$$

where $(1)$ follows from the application of the third identity~\eqref{WH} for $(2313)$, $(2)$ follows from the application of the third identity~\eqref{WH} for $(3123)$, $(3 )$ follows from~\ref{obrcom}. Note that $(1312)=(1332)$ and $(2321)=(2331)$, hence $(2313) \equiv (1312)^{-1}(2321)(1323)\mod\gamma_5$, and therefore any of the commutators $(1323)$ and $(2313)$ can be removed.

Similarly, $(1321)\equiv(1231)(2311) \mod \gamma_5$. Consider
$$(1231)\equiv(3211)(1321) \mod \gamma_5\equiv(3211)((12), (13))(1312) \mod \gamma_5.$$
At the same time, it is true
$$(3211)\equiv(2311)^{-1} \mod \gamma_5\equiv(1321)^{-1}(2131)^{-1} \mod \gamma_5\equiv$$
$$\equiv(1321)^{-1}(1231) \mod \gamma_5\equiv(1321)^{-1}((13), (12))(1213) \mod \gamma_5.$$
From here we get
$$(3211)((12), (13))(1312)\equiv(1321)^{-1}((13), (12))(1213)((12), (13))(1312) \mod \gamma_5 \equiv$$
$$\equiv(1321)^{-1}(1213)(1312) \mod \gamma_5.$$
As a result, we got that $(1321)\equiv(1321)^{-1}(1213)(1312)(2311) \mod \gamma_5$, hence $(1321)^2\equiv(1213)(1312) (2311) \mod \gamma_5$, but $(1321)^2\equiv1\mod\gamma_5$, so $(1213)(1312)(2311)\equiv1\mod \gamma_5$. We get that $(2311)\equiv(1312)^{-1}(1213)^{-1}\mod\gamma_5$. Moreover, $(1213)$ was expressed above in terms of $(1312)$ and $(2321)$. So we can remove $(2311)$.

Now consider $(1322)\equiv(1232)(2312)\mod\gamma_5$. Further,
$$(1232)\equiv(3212)(1322)\mod\gamma_5\equiv(3212)(2312)(1232)\mod\gamma_5\equiv$$
$$\equiv(3212)((21), (23))(2321)((23),(12))(1223)\mod\gamma_5\equiv$$
$$\equiv(3212)((21),(23))((21),(23))(2321)(1223)\mod\gamma_5\equiv$$
$$\equiv(3212)(2321)(1223)\mod\gamma_5.$$
From here we get
$$(1322)\equiv(3212)(2321)(1223)(2312)\mod\gamma_5\equiv(2321)(1223)=(2321)(1213).$$
Moreover, $(1213)$ was expressed above in terms of $(1312)$ and $(2321)$. So we can remove $(1322)$.

As a result, we got a set of $8$ commutators
$$
(g_1, g_2, g_1, g_1), (g_1, g_3, g_1, g_1), (g_2, g_3, g_2, g_1), (g_1, g_3, g_2, g_1), \cancel{(g_2, g_3, g_1, g_1)},
$$
$$
\cancel{(g_1, g_2, g_1, g_2)}, (g_1, g_3, g_1, g_2), (g_2, g_3, g_2, g_2), \cancel{(g_1, g_3, g_2, g_2)}, (g_2, g_3, g_1, g_2),
$$
$$
\cancel{(g_1, g_2, g_1, g_3)}, \cancel{(g_1, g_3, g_1, g_3)}, \cancel{(g_2, g_3, g_2, g_3)}, (g_1, g_3, g_2, g_3), \cancel{(g_2, g_3, g_1, g_3)},
$$
which minimally generates $\gamma_4(\racg_\sK)$. Hence, according to the suggestion~\ref{numbergens3}, the set
$$
[\mu_1, \mu_2, \mu_1, \mu_1], [\mu_1, \mu_3, \mu_1, \mu_1], [\mu_2, \mu_3, \mu_2, \mu_1], [\mu_1, \mu_3, \mu_2, \mu_1], [\mu_1, \mu_3, \mu_1, \mu_2],
$$
$$
[\mu_2, \mu_3, \mu_2, \mu_2], [\mu_2, \mu_3, \mu_1, \mu_2], [\mu_1, \mu_3, \mu_2, \mu_3]
$$
minimally generates $L^4(\racg_\sK)$.
\end{proof}

\begin{corollary}\label{cor_freeproduct3}
Let $\sK$ be a discrete set of $3$-th points, i.e. $\racg_\sK = \mathbb Z_2\langle g_1 \rangle \ast \mathbb Z_2\langle g_2 \rangle \ast \mathbb Z_2\langle g_3 \rangle$. Then $L^4(\racg_\sK)\cong\mathbb Z_2^8$ and has the following minimum set of generators:
$$
[\mu_j, \mu_i, \mu_i, \mu_i], [\mu_k, \mu_i, \mu_i, \mu_i], [\mu_k, \mu_j, \mu_j, \mu_i], [\mu_k, \mu_i, \mu_j, \mu_i], [\mu_k, \mu_i, \mu_i, \mu_j],
$$
$$
[\mu_k, \mu_j, \mu_j, \mu_j], [\mu_k, \mu_j, \mu_i, \mu_j], [\mu_k, \mu_i, \mu_j, \mu_k],
$$
where $i, j, k$ are any distinct numbers from $\{1, 2, 3\}$.
\end{corollary}

\begin{proof}
It follows from the proposition~\ref{freeproduct3} and the symmetry of the generators that in the indicated set one can assign different indices to different letters (in particular, $i = 1, j = 2, k = 3$), while in each commutator the first two elements can be swap. From this follows the proof.
\end{proof}


\begin{theorem}\label{commcox3}
Let $\sK$ be a simplicial complex on a set~$[3]$. Then
\begin{enumerate}[label=(\alph*)]
\item if in $\sK$ has all the edges $\{1, 2\}, \{1, 3\}, \{2, 3\}$, то $L^4(\racg_\sK) = \{e\}$, and there are no generators;
\item if in $\sK$ there are only $2$ edges $\{i, k\}, \{j, k\}$, where $i, j, k \in \{1, 2, 3\}$ and are pairwise distinct, and $i < j$, then $L^4(\racg_\sK) \cong \mathbb Z_2$ and is minimally generated by the element $[\mu_i, \mu_j, \mu_i, \mu_i]$;
\item  $\sK$ there is only $1$ edge $\{i, j\}$, where $i, j \in \{1, 2, 3\}, i < j$, then $L^4(\racg_\sK) \cong \mathbb Z_2^4$ and is minimally generated by four elements of the form
$$
[\mu_i, \mu_k, \mu_i, \mu_i], [\mu_k, \mu_j, \mu_k, \mu_k], [\mu_k, \mu_j, \mu_k, \mu_i], [\mu_k, \mu_j, \mu_i, \mu_k],
$$
where $k \neq i, k \neq j$;
\item if in $\sK$ no edges then $L^4(\racg_\sK) \cong \mathbb Z_2^8$ and is minimally generated by eight elements of the form
$$
[\mu_j, \mu_i, \mu_i, \mu_i], [\mu_k, \mu_i, \mu_i, \mu_i], [\mu_k, \mu_j, \mu_j, \mu_i], [\mu_k, \mu_i, \mu_j, \mu_i], [\mu_k, \mu_i, \mu_i, \mu_j],
$$
$$
[\mu_k, \mu_j, \mu_j, \mu_j], [\mu_k, \mu_j, \mu_i, \mu_j], [\mu_k, \mu_i, \mu_j, \mu_k],
$$
where $i, j, k$ are any distinct numbers from $\{1, 2, 3\}$.
\end{enumerate}
\end{theorem}

\begin{proof}
And suggestions~\ref{freeproduct3} for all items we have that set
$$
[\mu_1, \mu_2, \mu_1, \mu_1], [\mu_1, \mu_3, \mu_1, \mu_1], [\mu_2, \mu_3, \mu_2, \mu_1], [\mu_1, \mu_3, \mu_2, \mu_1], [\mu_1, \mu_3, \mu_1, \mu_2],
$$
$$
[\mu_2, \mu_3, \mu_2, \mu_2], [\mu_2, \mu_3, \mu_1, \mu_2], [\mu_1, \mu_3, \mu_2, \mu_3]
$$
generates $L^4(\racg_\sK)$ (not necessarily minimal, moreover, never minimally generates).

For item (a), in the presence of all edges, all elements of the set turn into $0$, whence the required result follows.

For item (b), consider the case when there are edges $\{1, 3\}$ and $\{2, 3\}$ (that is, $i = 1, j = 2, k = 3$). In this case, all commutators turn to $0$, except for $[\mu_1, \mu_2, \mu_1, \mu_1]$, which implies that the item is true for this case. Since it is always possible to change the numbering of vertices in any way, then for any two edges $L^4(\racg_\sK) \cong \mathbb Z_2$, and one can leave any non-zero commutator from the given set, which implies the desired.

For item (c), consider the case when $\sK$ has a single edge $\{1, 3\}$. Then there is a set of commutators
$$
[\mu_1, \mu_2, \mu_1, \mu_1], [\mu_2, \mu_3, \mu_2, \mu_2], [\mu_2, \mu_3, \mu_2, \mu_1], [\mu_2, \mu_3, \mu_1, \mu_2].
$$
This set is minimal according to suggestion~\ref{numbergens3oneedge}.

For simplicial complexes with a different edge, the argument is similar to the proof of Corollary~\ref{cor_freeproduct3}.

Item (d) follows from corollary~\ref{cor_freeproduct3}.
\end{proof}

\begin{proposition}\label{freeproduct4}
Let $\sK$ be a discrete set of $4$-th points, i.e. $\racg_\sK = \mathbb Z_2\langle g_1 \rangle \ast \mathbb Z_2\langle g_2 \rangle \ast \mathbb Z_2\langle g_3 \rangle \ast \mathbb Z_2\langle g_4 \rangle$. Then $L^4(\racg_\sK)\cong\mathbb Z_2^{32}$ has a minimal set of generators $\overline{A_1}\cup \overline{A_2}\cup \overline{A_3}\cup \overline{A_4}\cup \overline{B}$, where
$$
\overline{A_1} = \{ [\mu_3, \mu_2, \mu_2, \mu_1], [\mu_3, \mu_1, \mu_2, \mu_1], [\mu_3, \mu_1, \mu_1, \mu_2], [\mu_3, \mu_2, \mu_1, \mu_2], [\mu_3, \mu_1, \mu_2, \mu_3] \}.
$$
$$
\overline{A_2} = \{ [\mu_2, \mu_1, \mu_1, \mu_1], [\mu_4, \mu_2, \mu_2, \mu_1], [\mu_4, \mu_1, \mu_2, \mu_1],
$$
$$
[\mu_4, \mu_1, \mu_1, \mu_2], [\mu_4, \mu_2, \mu_1, \mu_2], [\mu_4, \mu_1, \mu_2, \mu_4] \},
$$
$$
\overline{A_3} = \{ [\mu_3, \mu_1, \mu_1, \mu_1], [\mu_4, \mu_1, \mu_1, \mu_1], [\mu_4, \mu_3, \mu_3, \mu_1], [\mu_4, \mu_1, \mu_3, \mu_1],
$$
$$
[\mu_4, \mu_1, \mu_1, \mu_3], [\mu_4, \mu_3, \mu_1, \mu_3], [\mu_4, \mu_1, \mu_3, \mu_4] \},
$$
$$
\overline{A_4} = \{ [\mu_3, \mu_2, \mu_2, \mu_2], [\mu_4, \mu_2, \mu_2, \mu_2], [\mu_4, \mu_3, \mu_3, \mu_2], [\mu_4, \mu_2, \mu_3, \mu_2],
$$
$$
[\mu_4, \mu_2, \mu_2, \mu_3], [\mu_4, \mu_3, \mu_3, \mu_3], [\mu_4, \mu_3, \mu_2, \mu_3], [\mu_4, \mu_2, \mu_3, \mu_4] \},
$$
$$
\overline{B} = \{ [\mu_2, \mu_4, \mu_3, \mu_1], [\mu_1, \mu_4, \mu_3, \mu_2], [\mu_1, \mu_4, \mu_2, \mu_3], [\mu_2, \mu_4, \mu_1, \mu_3],
$$
$$
[\mu_3, \mu_4, \mu_1, \mu_2], [\mu_3, \mu_4, \mu_2, \mu_1] \}.
$$
\end{proposition}

\begin{proof}
Consider commutators: $(g_i, g_j, g_i)=(g_j, g_i)g_i(g_i, g_j)g_i=g_jg_ig_jg_ig_ig_ig_jg_ig_jg_i=g_jg_ig_jg_ig_jg_ig_jg_i=(g_j, g_i)^2, (g_i, g_j, g_j)=(g_j, g_i)g_j(g_i, g_j)g_j=g_jg_ig_jg_ig_jg_ig_jg_ig_jg_j=g_jg_ig_jg_ig_jg_ig_jg_i=(g_j, g_i)^2$, that is $(g_i, g_j, g_i)=(g_i, g_j, g_j)$.

By Theorem~\ref{LRCK} the group $\gamma_3(\racg_\sK)/\gamma_4(\racg_\sK)$ is minimally generated by commutators
$$
(g_1, g_2, g_2), (g_1, g_3, g_3), (g_1, g_4, g_4), (g_2, g_3, g_3), (g_2, g_4, g_4), (g_3, g_4, g_4),
$$
$$
(g_1, g_3, g_2), (g_1, g_4, g_2), (g_1, g_4, g_3), (g_2, g_4, g_1), (g_2, g_4, g_3),
$$
$$
(g_3, g_4, g_1), (g_3, g_4, g_2), (g_2, g_3, g_1).
$$
Hence, the group $\gamma_4(\racg_\sK)/\gamma_5(\racg_\sK)$ is generated (minimally) by commutators of the form $(z, g_i)$, where $z$ is the commutator from the set above and $i \in \{1, 2, 3, 4\}$.

Let us write them out, dividing them into groups as follows: in the group $A_1$ there will be commutators with indices $\{1, 2, 3\}$, in the group $A_2$ there will be all commutators with indices $\{1, 2, 4\} $, not included in $A_1$, in the group $A_3$ there will be commutators with indices $\{1, 3, 4\}$, not included in $A_1 \cup A_2$, in the group $A_4$ there will be commutators with indices $ \{2, 3, 4\}$ not included in $A_1 \cup A_2 \cup A_3$, in group B there will be commutators with $4$-indexes:
$$
A_1 =\{ (g_1, g_2, g_2, g_1), (g_1, g_2, g_2, g_2), (g_1, g_3, g_2, g_3), (g_2, g_3, g_3, g_3), (g_2, g_3, g_1, g_3),
$$
$$
(g_1, g_3, g_3, g_1), (g_2, g_3, g_3, g_1), (g_1, g_2, g_2, g_3), (g_1, g_3, g_2, g_1), (g_2, g_3, g_1, g_1),
$$
$$
(g_1, g_3, g_3, g_2), (g_2, g_3, g_3, g_2), (g_1, g_3, g_2, g_2), (g_1, g_3, g_3, g_3), (g_2, g_3, g_1, g_2) \},
$$
$$
A_2 = \{ (g_1, g_2, g_2, g_4), (g_1, g_4, g_4, g_1), (g_2, g_4, g_4, g_1), (g_1, g_4, g_2, g_1), (g_2, g_4, g_1, g_1),
$$
$$
(g_1, g_4, g_4, g_2), (g_2, g_4, g_4, g_2), (g_1, g_4, g_2, g_4), (g_2, g_4, g_4, g_4), (g_1, g_4, g_2, g_2),
$$
$$
(g_2, g_4, g_1, g_2), (g_2, g_4, g_1, g_4), (g_1, g_4, g_4, g_4) \},
$$
$$
A_3 = \{ (g_1, g_4, g_4, g_3), (g_3, g_4, g_4, g_3), (g_1, g_4, g_3, g_3), (g_3, g_4, g_4, g_1), (g_1, g_4, g_3, g_1),
$$
$$
(g_1, g_4, g_3, g_4), (g_1, g_3, g_3, g_4), (g_3, g_4, g_4, g_4), (g_3, g_4, g_1, g_1), (g_3, g_4, g_1, g_3),
$$
$$
(g_3, g_4, g_1, g_4) \},
$$
$$
A_4 = \{ (g_2, g_4, g_3, g_4), (g_3, g_4, g_4, g_2), (g_2, g_4, g_3, g_2), (g_2, g_4, g_4, g_3), (g_2, g_4, g_3, g_3),
$$
$$
(g_2, g_3, g_3, g_4), (g_3, g_4, g_2, g_2), (g_3, g_4, g_2, g_3), (g_3, g_4, g_2, g_4) \},
$$
$$
B = \{ (g_2, g_4, g_3, g_1), (g_1, g_4, g_3, g_2), (g_1, g_4, g_2, g_3), (g_2, g_4, g_1, g_3), (g_1, g_3, g_2, g_4),
$$
$$
(g_3, g_4, g_1, g_2), (g_3, g_4, g_2, g_1), (g_2, g_3, g_1, g_4) \}.
$$
Since $A_1$ contains only commutators with indices $1, 2$ and $3$, then all commutators from this set can be replaced by commutators from Corollary~\ref{cor_freeproduct3} (we take $i = 1, j = 2, k = 3 $), getting the set $A_1^{(1)}$:
$$
A_1^{(1)} = \{ (g_2, g_1, g_1, g_1), (g_3, g_1, g_1, g_1), (g_3, g_2, g_2, g_1), (g_3, g_1, g_2, g_1),
$$
$$
(g_3, g_1, g_1, g_2), (g_3, g_2, g_2, g_2), (g_3, g_2, g_1, g_2), (g_3, g_1, g_2, g_3) \}.
$$
Now we transfer to $A_2$ commutators without index $3$ from $A_1^{(1)}$ and do the same, getting $A_2^{(1)}$ (take $i = 1, j = 2, k = 4$). The resulting new set of commutators from $A_1^{(1)}$ will be denoted by $A_1^{(2)}$:
$$
A_1^{(2)} = \{ (g_3, g_1, g_1, g_1), (g_3, g_2, g_2, g_1), (g_3, g_1, g_2, g_1),
$$
$$
(g_3, g_1, g_1, g_2), (g_3, g_2, g_2, g_2), (g_3, g_2, g_1, g_2), (g_3, g_1, g_2, g_3) \}.
$$
$$
A_2^{(1)} = \{ (g_2, g_1, g_1, g_1), (g_4, g_1, g_1, g_1), (g_4, g_2, g_2, g_1), (g_4, g_1, g_2, g_1),
$$
$$
(g_4, g_1, g_1, g_2), (g_4, g_2, g_2, g_2), (g_4, g_2, g_1, g_2), (g_4, g_1, g_2, g_4) \}.
$$
Similarly, we transfer commutators from $A_1^{(2)}$ and from $A_2^{(1)}$ without index $2$ to $A_3$, then we make a replacement, getting $A_3^{(1)}$ (we take $i = 1, j = 3, k = 4$). The resulting new sets will be denoted by $A_1^{(3)}$ and $A_2^{(2)}$:
$$
A_1^{(3)} = \{ (g_3, g_2, g_2, g_1), (g_3, g_1, g_2, g_1),
$$
$$
(g_3, g_1, g_1, g_2), (g_3, g_2, g_2, g_2), (g_3, g_2, g_1, g_2), (g_3, g_1, g_2, g_3) \}.
$$
$$
A_2^{(2)} = \{ (g_2, g_1, g_1, g_1), (g_4, g_2, g_2, g_1), (g_4, g_1, g_2, g_1),
$$
$$
(g_4, g_1, g_1, g_2), (g_4, g_2, g_2, g_2), (g_4, g_2, g_1, g_2), (g_4, g_1, g_2, g_4) \},
$$
$$
A_3^{(1)} = \{ (g_3, g_1, g_1, g_1), (g_4, g_1, g_1, g_1), (g_4, g_3, g_3, g_1), (g_4, g_1, g_3, g_1),
$$
$$
(g_4, g_1, g_1, g_3), (g_4, g_3, g_3, g_3), (g_4, g_3, g_1, g_3), (g_4, g_1, g_3, g_4) \}.
$$
Next, we transfer to $A_4$ commutators without index $1$ from $A_i^{(j)}, i, j > 0, i + j = 4$, and then we make a similar replacement, obtaining $A_4^{(1)}$ (we take $i = 2, j = 3, k = 4$). The resulting new sets will be denoted by $A_1^{(4)}$, $A_2^{(3)}$ and $A_3^{(2)}$ (and additionally, as indicated below):
$$
A_1' = A_1^{(4)} = \{ (g_3, g_2, g_2, g_1), (g_3, g_1, g_2, g_1),
$$
$$
(g_3, g_1, g_1, g_2), (g_3, g_2, g_1, g_2), (g_3, g_1, g_2, g_3) \}.
$$
$$
A_2' = A_2^{(3)} = \{ (g_2, g_1, g_1, g_1), (g_4, g_2, g_2, g_1), (g_4, g_1, g_2, g_1),
$$
$$
(g_4, g_1, g_1, g_2), (g_4, g_2, g_1, g_2), (g_4, g_1, g_2, g_4) \},
$$
$$
A_3' = A_3^{(2)} = \{ (g_3, g_1, g_1, g_1), (g_4, g_1, g_1, g_1), (g_4, g_3, g_3, g_1), (g_4, g_1, g_3, g_1),
$$
$$
(g_4, g_1, g_1, g_3), (g_4, g_3, g_1, g_3), (g_4, g_1, g_3, g_4) \},
$$
$$
A_4' = A_4^{(1)} = \{ (g_3, g_2, g_2, g_2), (g_4, g_2, g_2, g_2), (g_4, g_3, g_3, g_2), (g_4, g_2, g_3, g_2),
$$
$$
(g_4, g_2, g_2, g_3), (g_4, g_3, g_3, g_3), (g_4, g_3, g_2, g_3), (g_4, g_2, g_3, g_4) \}.
$$
The set $A_1'\cup A_2'\cup A_3'\cup A_4'\cup B$, consisting of $34$ commutators, generates (according to the proposition~\ref{numbergens4}, non-minimal) the group $\gamma_4(\racg_\sK) /\gamma_5(\racg_\sK)$. Further in this proof, for convenience, we will use the notation $(g_{i_1}, g_{i_2}, \ldots, g_{i_k}) = (i_1i_2i_3\ldots i_k)$, and instead of $\gamma_k(\racg_\sK) $ will be written $\gamma_k$, and all modulo comparisons will default to $\gamma_5$, unless otherwise specified.

Consider $(2314)^{-1}\equiv(3214)\equiv(1234)(1324)^{-1}$. Note that $(1423\equiv(1243)(2413)\equiv(34,12)(1234)(2413)\equiv^{(1)}(3412)(3421)^{-1}(1234)( 2413)$, where $(1)$ follows from the fact that $(3412)\equiv(21,34)(3421)$ From this we get that $(1234)\equiv(3421)(3412)^{- 1}(1423)(2413)^{-1}$, which means
$$
(2314)\equiv(1324)(1234)^{-1}\equiv(1324)(2413)(1423)^{-1}(3412)(3421)^{-1}.
$$
We get that $(2314)$ can be removed, since it is expressed in terms of other commutators.

Since $(2431)\equiv(13,24)(2413)$, then $(24,31)\equiv(2431)(2413)^{-1}$. Hence we get that
$$
(3412)\equiv(1432)(3142)\equiv(1432)(24,31)(3124)\equiv(1432)(2431)(2413)^{-1}(1324)^{-1}.
$$
This means that the commutator $(1324)$ can be removed, since it is expressed in terms of other commutators. We get commutators:
$$
A_1' = \{ (g_3, g_2, g_2, g_1), (g_3, g_1, g_2, g_1), (g_3, g_1, g_1, g_2), (g_3, g_2, g_1, g_2), (g_3, g_1, g_2, g_3) \}.
$$
$$
A_2' = \{ (g_2, g_1, g_1, g_1), (g_4, g_2, g_2, g_1), (g_4, g_1, g_2, g_1),
$$
$$
(g_4, g_1, g_1, g_2), (g_4, g_2, g_1, g_2), (g_4, g_1, g_2, g_4) \},
$$
$$
A_3' = \{ (g_3, g_1, g_1, g_1), (g_4, g_1, g_1, g_1), (g_4, g_3, g_3, g_1), (g_4, g_1, g_3, g_1),
$$
$$
(g_4, g_1, g_1, g_3), (g_4, g_3, g_1, g_3), (g_4, g_1, g_3, g_4) \},
$$
$$
A_4' = \{ (g_3, g_2, g_2, g_2), (g_4, g_2, g_2, g_2), (g_4, g_3, g_3, g_2), (g_4, g_2, g_3, g_2),
$$
$$
(g_4, g_2, g_2, g_3), (g_4, g_3, g_3, g_3), (g_4, g_3, g_2, g_3), (g_4, g_2, g_3, g_4) \}.
$$
$$
B = \{ (g_2, g_4, g_3, g_1), (g_1, g_4, g_3, g_2), (g_1, g_4, g_2, g_3), (g_2, g_4, g_1, g_3),
$$
$$
(g_3, g_4, g_1, g_2), (g_3, g_4, g_2, g_1) \}.
$$
The resulting collection contains $32$ commutators, which, by the proposition~\ref{numbergens4}, minimally generates the group $\gamma_4/\gamma_5$. It follows from this that the set $\overline{A_1}\cup \overline{A_2}\cup \overline{A_3}\cup \overline{A_4}\cup \overline{B}$, where
$$
\overline{A_1} = \{ [\mu_3, \mu_2, \mu_2, \mu_1], [\mu_3, \mu_1, \mu_2, \mu_1], [\mu_3, \mu_1, \mu_1, \mu_2], [\mu_3, \mu_2, \mu_1, \mu_2], [\mu_3, \mu_1, \mu_2, \mu_3] \}.
$$
$$
\overline{A_2} = \{ [\mu_2, \mu_1, \mu_1, \mu_1], [\mu_4, \mu_2, \mu_2, \mu_1], [\mu_4, \mu_1, \mu_2, \mu_1],
$$
$$
[\mu_4, \mu_1, \mu_1, \mu_2], [\mu_4, \mu_2, \mu_1, \mu_2], [\mu_4, \mu_1, \mu_2, \mu_4] \},
$$
$$
\overline{A_3} = \{ [\mu_3, \mu_1, \mu_1, \mu_1], [\mu_4, \mu_1, \mu_1, \mu_1], [\mu_4, \mu_3, \mu_3, \mu_1], [\mu_4, \mu_1, \mu_3, \mu_1],
$$
$$
[\mu_4, \mu_1, \mu_1, \mu_3], [\mu_4, \mu_3, \mu_1, \mu_3], [\mu_4, \mu_1, \mu_3, \mu_4] \},
$$
$$
\overline{A_4} = \{ [\mu_3, \mu_2, \mu_2, \mu_2], [\mu_4, \mu_2, \mu_2, \mu_2], [\mu_4, \mu_3, \mu_3, \mu_2], [\mu_4, \mu_2, \mu_3, \mu_2],
$$
$$
[\mu_4, \mu_2, \mu_2, \mu_3], [\mu_4, \mu_3, \mu_3, \mu_3], [\mu_4, \mu_3, \mu_2, \mu_3], [\mu_4, \mu_2, \mu_3, \mu_4] \},
$$
$$
\overline{B} = \{ [\mu_2, \mu_4, \mu_3, \mu_1], [\mu_1, \mu_4, \mu_3, \mu_2], [\mu_1, \mu_4, \mu_2, \mu_3], [\mu_2, \mu_4, \mu_1, \mu_3],
$$
$$
[\mu_3, \mu_4, \mu_1, \mu_2], [\mu_3, \mu_4, \mu_2, \mu_1] \},
$$
minimally generates $L^4(\racg_\sK)$.
\end{proof}

this implies

\begin{theorem}\label{theorem_freeproduct4}
Let $\sK$ be a discrete set of $4$-th points, i.e. $\racg_\sK = \mathbb Z_2\langle g_1 \rangle \ast \mathbb Z_2\langle g_2 \rangle \ast \mathbb Z_2\langle g_3 \rangle \ast \mathbb Z_2\langle g_4 \rangle$. Then $L^4(\racg_\sK)\cong\mathbb Z_2^{32}$ has a minimal set of generators $\overline{A_i}\cup \overline{A_j}\cup \overline{A_k}\cup \overline{A_l}\cup \overline{B}$, where
$$
\overline{A_i} = \{ [\mu_k, \mu_j, \mu_j, \mu_i], [\mu_k, \mu_i, \mu_j, \mu_i], [\mu_k, \mu_i, \mu_i, \mu_j], [\mu_k, \mu_j, \mu_i, \mu_j], [\mu_k, \mu_i, \mu_j, \mu_k] \}.
$$
$$
\overline{A_j} = \{ [\mu_j, \mu_i, \mu_i, \mu_i], [\mu_l, \mu_j, \mu_j, \mu_i], [\mu_l, \mu_i, \mu_j, \mu_i], [\mu_l, \mu_i, \mu_i, \mu_j], [\mu_l, \mu_j, \mu_i, \mu_j],
$$
$$
[\mu_l, \mu_i, \mu_j, \mu_l] \},
$$
$$
\overline{A_k} = \{ [\mu_k, \mu_i, \mu_i, \mu_i], [\mu_l, \mu_i, \mu_i, \mu_i], [\mu_l, \mu_k, \mu_k, \mu_i], [\mu_l, \mu_i, \mu_k, \mu_i], [\mu_l, \mu_i, \mu_i, \mu_k],
$$
$$
[\mu_l, \mu_k, \mu_i, \mu_k], [\mu_l, \mu_i, \mu_k, \mu_l] \},
$$
$$
\overline{A_l} = \{ [\mu_k, \mu_j, \mu_j, \mu_j], [\mu_l, \mu_j, \mu_j, \mu_j], [\mu_l, \mu_k, \mu_k, \mu_j], [\mu_l, \mu_j, \mu_k, \mu_j], [\mu_l, \mu_j, \mu_j, \mu_k],
$$
$$
[\mu_l, \mu_k, \mu_k, \mu_k], [\mu_l, \mu_k, \mu_j, \mu_k], [\mu_l, \mu_j, \mu_k, \mu_l] \},
$$
$$
\overline{B} = \{ [\mu_j, \mu_l, \mu_k, \mu_i], [\mu_i, \mu_l, \mu_k, \mu_j], [\mu_i, \mu_l, \mu_j, \mu_k], [\mu_j, \mu_l, \mu_i, \mu_k], [\mu_k, \mu_l, \mu_i, \mu_j],
$$
$$
[\mu_k, \mu_l, \mu_j, \mu_i] \}.
$$
\end{theorem}

\begin{proof}
It follows from the proposition~\ref{freeproduct4} and the symmetry of the generators that in the indicated set one can assign different indices to different letters (in particular, $i = 1, j = 2, k = 3, l = 4$), and in each commutator the first two elements can be swapped. From this follows the proof.
\end{proof}

\end{document}